\documentclass[11pt,a4paper]{article}

\usepackage[T1]{fontenc}
\usepackage[utf8]{inputenc}
\usepackage{amsmath}
\usepackage{amssymb}
\usepackage{amsthm}
\usepackage{placeins}  % for \FloatBarrier
\usepackage{authblk}

\usepackage{graphicx}
\usepackage{epstopdf}
\def\am{\operatorname{am}}
\def\Cat{\operatorname{Cat}}
\def\mad{\operatorname{mad}}
\newtheorem{thm}{Theorem}
\newtheorem{lem}[thm]{Lemma}
\newtheorem{conj}[thm]{Conjecture}
\newtheorem{cor}[thm]{Corollary}
\newtheorem{prop}[thm]{Proposition}

\title{On the precise value of the strong chromatic index of a planar graph
       with a large girth}
\author[1,2]{Gerard Jennhwa Chang\thanks{E-mail: gjchang@math.ntu.edu.tw.}}
\author[1]{Guan-Huei Duh\thanks{E-mail: r03221028@ntu.edu.tw.}}
\affil[1]{Department of Mathematics, National Taiwan University, Taipei 10617, Taiwan}
\affil[2]{National Center for Theoretical Sciences, Mathematics Division,
2F of Astronomy-Mathematics Building, National Taiwan University, Taipei 10617, Taiwan}
\date{September 24, 2015} % Duh
\begin{document}

\maketitle

\begin{abstract}
A {\em strong $k$-edge-coloring} of a graph $G$ is a mapping
from $E(G)$ to $\{1,2,\ldots,k\}$ such that every pair of distinct
edges at distance at most two receive different colors.
The {\it strong chromatic index} $\chi'_s(G)$ of a graph $G$
is the minimum $k$ for which $G$ has a strong $k$-edge-coloring.
Denote $\sigma(G)=\max_{xy\in E(G)}\{\deg(x)+\deg(y)-1\}$.
It is easy to see that $\sigma(G) \le \chi'_s(G)$ for any graph $G$,
and the equality holds when $G$ is a tree.
For a planar graph $G$ of maximum degree $\Delta$,
it was proved that $\chi'_s(G) \le 4 \Delta +4$ by using the Four Color Theorem.
The upper bound was then reduced to $4\Delta$, $3\Delta+5$, $3\Delta+1$, $3\Delta$, $2\Delta-1$
under different conditions for $\Delta$ and the girth.
In this paper, we prove that if the girth of a planar graph $G$
is large enough and $\sigma(G)\geq \Delta(G)+2$,
then the strong chromatic index of $G$ is precisely $\sigma(G)$. This result
reflects the intuition that a planar graph with a large girth locally looks like a tree.

\end{abstract}

%\newpage
%\linenumbers
%%%%%%%%%%%%%%%%%%%%%%%%%%%%%%%%%%%%%%%%%%%%%%%%%%%%%%%%%%%%%%
\section{Introduction} % sec 1
%%%%%%%%%%%%%%%%%%%%%%%%%%%%%%%%%%%%%%%%%%%%%%%%%%%%%%%%%%%%%%

A {\em strong $k$-edge-coloring} of a graph $G$ is a mapping
from $E(G)$ to $\{1,2,\ldots,k\}$ such that every pair of distinct
edges at distance at most two receive different colors. It
induces a proper vertex coloring of $L(G)^{2}$, the square
of the line graph of $G$. The {\em strong chromatic index $\chi_{s}'(G)$}
of $G$ is the minimum $k$ for which $G$ has a strong $k$-edge-coloring.
This concept was introduced by Fouquet and Jolivet~\cite{FJ83,fj1984} 
to model the channel assignment in some radio networks.
For more applications, see~\cite{BI+06,NK+00,Ram97,RL99,JN2001,TW+2004}.

\medskip{}

A Vizing-type problem was asked by Erd\H{o}s and Ne\v{s}et\v{r}il,
and further strengthened by Faudree, Schelp, Gy\'{a}rf\'{a}s and Tuza
to give an upper bound for $\chi_{s}'(G)$ in terms of the maximum degree $\Delta=\Delta(G)$:

\begin{conj}[Erd\H{o}s and Ne\v{s}et\v{r}il~'88~\cite{e1988}~'89~\cite{en1989}, Faudree~{\em et al}~'90~\cite{FG+90}] \label{conjErdosNesetril}
If $G$ is a graph with maximum degree $\Delta$, then
$
  \chi_{s}'(G)\leq \Delta^2 + \lfloor\frac{\Delta}{2}\rfloor^2.
$
\end{conj}

\medskip{}

 As demonstrated in~\cite{FG+90}, there are indeed
 some graphs reach the given upper bounds.

\medskip{}

By a greedy algorithm, it can be easily seen that $\chi_{s}'(G)\leq2\Delta(\Delta-1)+1$.
Molloy and Reed~\cite{MR97} using the probabilistic method
to show that $\chi_{s}'(G)\le1.998\Delta^{2}$ for
maximum degree $\Delta$ large enough. Recently, this upper bound was improved by
Bruhn and Joos~\cite{BJ15} to $1.93\Delta^2$.

For small maximum degrees, the cases $\Delta=3$ and $4$ were studied.
Andersen~\cite{And92} and Hor\'ak {\em et al}~\cite{HQT93} proved that
$\chi'_s(G)\leq 10$ for $\Delta(G)\leq 3$ independently; and
Cranston~\cite{Cra06} showed that $\chi'_s(G)\leq 22$ when $\Delta(G)\leq 4$.

According to the examples in~\cite{FG+90}, the bound
is tight for $\Delta=3$, and the best we may expect for $\Delta=4$ is 20.

\medskip{}

The strong chromatic index of a few families of graphs are examined, such as cycles, trees,
$d$-dimensional cubes, chordal graphs,
Kneser graphs, $k$-degenerate graphs, chordless graphs and $C_4$-free graphs,
see~\cite{BF2015,CN2013,DG+2015,FG+90,Mah00,Wang2015,Yu2015}.  As for Halin graphs,
refer to \cite{cL,LLt2010,LL2011,slt2006,st2009}.
For the relation to various graph products, see~\cite{Tog2007}.

\medskip{}

Now we turn to planar graphs.

Faudree {\em et al} used the Four Color Theorem~\cite{AHa77,AHb77} to prove that
planar graphs with maximum degree $\Delta$ are strong $(4\Delta+4)$-edge-colorable~\cite{FG+90}.
By the same spirit, it can be shown that $K_{5}$-minor free graphs are
strong $(4\Delta+4)$-edge-colorable. Moreover,
every planar $G$ with girth
at least 7 and $\Delta\ge 7$ is strong $3\Delta$-edge-colorable by applying
a strengthened version of Vizing's Theorem on planar graphs~\cite{SanZ2001,Viz65}
and Gr\H{o}tzsch's theorem~\cite{Gro59}.

\medskip{}

%We may expect to get a smaller upper
%bound if we raise the girth of a planar graph to be larger.
The following results are obtained by using a discharging method:

\begin{thm}[Hud\'ak {\em et al} '14~\cite{Hudak+14}] \label{theo-Hudak3D+5}
If $G$ is a planar graph with girth at least $6$ and maximum degree
at least $4$, then $\chi'_{s}(G)\leq3\Delta(G)+5$.
\end{thm}

\begin{thm}[Hud\'ak {\em et al} '14~\cite{Hudak+14}] \label{theo-Hudak3D}
If $G$ is a planar graph with girth at least $7$, then $\chi'_{s}(G)\leq3\Delta(G)$.
\end{thm}

\medskip{}

And the bounds are improved by Bensmail {\em et al}.

\begin{thm}[Bensmail {\em et al} '14~\cite{Bensmail+14}]
\label{theo-Bensmail} If $G$ is a planar graph with girth at least
$6$, then $\chi'_{s}(G)\leq3\Delta(G)+1$. \end{thm}

\begin{thm}[Bensmail {\em et al} '14~\cite{Bensmail+14}]
\label{theo-Bensmail-1} If $G$ is a planar graph with girth at least
$5$ or maximum degree at least $7$, then $\chi'_{s}(G)\leq4\Delta(G)$.
\end{thm}

\medskip{}

It is also interesting to see the asymptotic behavior of strong chromatic
index when the girth is large enough.

\medskip{}

\begin{thm}[Borodin and Ivanova '13~\cite{Borodin+13}]
If
$G$ is a planar graph with maximum degree $\Delta \ge 3$ and girth
at least $40\lfloor\frac{\Delta}{2}\rfloor+1$, then $\chi'_{s}(G)\leq2\Delta-1$.
\end{thm}

\begin{thm}[Chang {\em et al} '13~\cite{Chang+13}] If $G$
is a planar graph with maximum degree $\Delta \ge 4$ and girth at least
$10\Delta+46$, then $\chi'_{s}(G)\leq2\Delta-1$.
\end{thm}

\begin{thm}[Wang and Zhao '15~\cite{WZ15}] If $G$
is a planar graph with maximum degree $\Delta \ge 4$ and girth at least
$10\Delta-4$, then $\chi'_{s}(G)\leq2\Delta-1$.
\end{thm}

\medskip{}

The concept of maximum average degree is also
an indicator to the sparsity of a graph.
Graphs with small maximum average degrees are in relation to
planar graphs with large girths, %to a large extent,
as a folklore lemma that can be proved
by Euler's formula points out.

\begin{lem}
A planar graph $G$ with girth $g$ has maximum average degree $\mad(G)<2+\frac{4}{g-2}$.
\end{lem}

Many results concerning planar graphs with large girths can be extended to general graphs with
small maximum average degrees and large girths. Strong chromatic index is no exception.

\begin{thm}[Wang and Zhao '15~\cite{WZ15}] \label{mad_WZ} Let $G$
be a graph with maximum degree $\Delta \ge 4$.
If the maximum average degree $\mad(G)<2+\frac{1}{3\Delta-2}$,
the even girth is at least $6$ and the odd girth is at least
$2\Delta-1$, then $\chi'_{s}(G)\leq2\Delta-1$.
\end{thm}

In terms of maximum degree $\Delta$, the bound $2\Delta-1$ is best
possible. We seek for a better parameter as a refinement. Define
\[
\sigma(G):=\max_{xy\in E(G)}\{\deg(x)+\deg(y)-1\}.
\]

An {\em antimatching} is an edge set $S\subseteq E(G)$ in which
any two edges are at distance at most $2$, thus any strong edge-coloring
assigns distinct colors on $S$. Notice that each color set of a strong
edge-coloring is an induced matching, and the intersection of an induced
matching and an antimatching contains at most one edge. The fact suggests
a dual problem to strong edge-coloring: finding a maximum antimatching
of $G$, whose size is denoted by $\am(G)$. For any edge $xy\in E(G)$,
the edges incident with $xy$ form an antimatching of size $\deg(x)+\deg(y)-1$.
Together with the weak duality, this gives the inequality
\[
\chi_{s}'(G)\geq\am(G)\geq\sigma(G).
\]
By induction, we see that for any nontrivial tree $T$, $\chi'_{s}(T)=\sigma(T)$
attains the lower bound~\cite{FG+90}. Based on the intuition that
a planar graph with large girth locally looks like a tree, in this
paper, we focus on this class of graphs.
More precisely, we prove the following main theorem:

%------------------------------------------------------
\begin{thm} \label{main}
If $G$ is a planar graph with $\sigma=\sigma(G)\geq 5$,
$\sigma\geq\Delta(G)+2$ and girth at least $5\sigma+16$,
then $\chi_{s}'(G)=\sigma$.
\end{thm}

We also make refinement on the girth constraint and gain a stronger result
in Section~\ref{sec_refine}.

\medskip{}

The condition $\sigma\geq\Delta(G)+2$ is necessary as shown in the following example.
Suppose $n \ge 1$ and $d \ge 2$.
Construct $G_{3n+1,d}$ from the cycle $(x_1,x_2, \ldots, x_{3n+1})$
by adding $d-2$ leaves adjacent to each $x_{3i}$ for $1 \le i \le n$.
Then $\sigma(G_{3n+1,d})=d+1<d+2=\Delta(G_{3n+1,d})+2$. 
See Figure~\ref{fig:counter} for $G_{3n+1,4}$.

\begin{figure}[!ht]{
\centering
\includegraphics{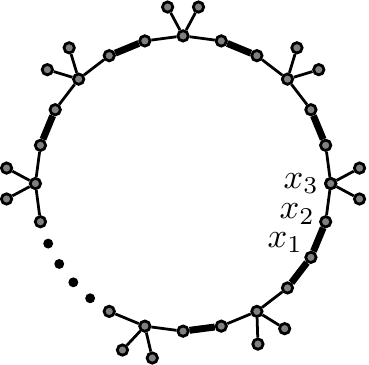}
\vskip -0.3cm
\caption{The graph $G_{3n+1,4}$.}
\label{fig:counter}
}
\end{figure}

We claim that $\sigma(G_{3n+1,d})<\chi'_s(G_{3n+1,d})$.
Suppose to the contrary that $\sigma(G_{3n+1,d})=\chi'_s(G_{3n+1,d})$.
For $1\le i \le n$, the $\sigma-1$ edges incident to $x_{3i}$,
together with $x_{3i-2} x_{3i-1}$ (or $x_{3i+1} x_{3i+2}$) use all the $\sigma$ colors,
implying that $x_{3i-2} x_{3i-1}$ uses the same color as $x_{3i+1} x_{3i+2}$,
where $x_{3n+2}=x_1$.
Therefore, $x_1x_2, x_4x_5, \ldots,x_{3n+1}x_{3n+2}$ all use the same color,
contradicting that $x_1x_2$ is adjacent to $x_{3n+1}x_1=x_{3n+1}x_{3n+2}$.

%%%%%%%%%%%%%%%%%%%%%%%%%%%%%%%%%%%%%%%%%%%%%%%%%%%%%%%%
\section{The proof of the main theorem}
%%%%%%%%%%%%%%%%%%%%%%%%%%%%%%%%%%%%%%%%%%%%%%%%%%%%%%%%

To prove the main theorem, we need two lemmas and a key lemma
(Lemma \ref{caterpillar}) to be verified in the next section.

The first lemma can be used to prove that any tree $T$ has
strong chromatic index $\sigma(T)$ by induction.

%------------------------------------------------------------
\begin{lem} \label{cut}
Suppose $x_{1}x_{2}$ is a cut edge of a graph $G$,
and $G_{i}$ is the component of $G-x_{1}x_{2}$
containing $x_{i}$ joining the edge $x_{1}x_{2}$ for $i=1,2$.
If for some integer
$k$, $\deg(x_{1})+\deg(x_{2})-1\leq k$ and $\chi_{s}'(G_{i})\leq k$ for
$i=1,2$, then $\chi_{s}'(G)\leq k$.
\end{lem}

\begin{proof}
Choose a strong $k$-edge-coloring $f_i$ of $G_{i}$ for $i=1,2$.
Let $E_{i}$ be the set of edges incident with $x_{i}$ in $G_{i}-x_{1}x_{2}$
and $S_{i}=f_{i}(E_{i})$.
Since $\deg(x_{1})+\deg(x_{2})-1\leq k$, we may assume
$S_1$ and $S_{2}$ are disjoint and $f_1(x_1x_2)=f_2(x_1x_2)$ is some
element $c \in \{1,2,\dots,k\}\backslash\left(S_1 \cup S_{2}\right)$.
Then
\[
f(e)=\begin{cases}
f_{1}(e), & \mbox{ if }e\in E(G_{1})-x_{1}x_{2};\\
f_{2}(e), & \mbox{ if }e\in E(G_{2})-x_{1}x_{2};\\
c,        & \mbox{ if }e=x_{1}x_{2}
\end{cases}
\]
is a strong $k$-edge-coloring of $G$.
\end{proof}

The following lemma from~\cite{NRS1997} about planar graphs is also useful
in the proof of the main theorem.
An $\ell$-thread is an induced path of $\ell+2$ vertices all of
whose internal vertices are of degree $2$ in the full graph.

%----------------------------------------------------------------
\begin{lem} \label{p_lg_path}
Any planar graph $G$ with minimum degree at least $2$ and with girth at least
$5\ell+1$ contains an $\ell$-thread.
\end{lem}

\begin{proof}
Contract all the vertices of degree 2 to obtain $G'$.
% First $G'$ has no multi-edge,
% otherwise there is an induced path of $\lfloor\frac{5\ell+1}{2}\rfloor\geq \ell+2$
% vertices all of whose internal vertices are of degree 2.
  Notice that $G'$ is a planar graph which may have multi-edges and may be disconnected.
Embed $G'=(V,E)$ in the plane as $P$. Then Euler's Theorem says that
$|V|-|E|+|F| \ge 2$, where $F$ is the set of faces of $P$. If $G'$ has girth larger
than $5$, we have
$2|E|=\sum_{f\in F} \deg (f)\geq 6|F|$. But that $G'$ has no vertices of degree 2
implies $2|E| = \sum_{v\in V} \deg (v) \geq 3|V|$. Combining all these produces a
contradiction:
\[
   2 \le |V|-|E|+|F| \leq \frac{2}{3}|E|-|E|+\frac{1}{3}|E| = 0.
\]
Hence $G'$ has a cycle of length at most 5. The corresponding cycle in
$G$ has length at least $5\ell+1$. Thus one of these edges in $G'$ is contracted from
$\ell$ vertices in $G$, and so $G$ has the required path.
\end{proof}

\medskip{}

These two lemmas, together with a key lemma to be verified in the next section,
lead to the following proof of the main theorem:

\medskip{}

\noindent
\textbf{Proof of Theorem \ref{main}.}
Since the inequality $\chi_{s}'(G)\geq\sigma(G)$ is trivial, it
suffices to show that $\chi_{s}'(G)\leq\sigma(G)$. That is, $G$ admits
a strong $\sigma$-edge-coloring $\varphi$.
Suppose to the contrary that there is a counterexample $G$
with minimum vertex number.
Then there is no vertex $x$ adjacent
to $\deg(x)-1$ vertices of degree 1. For otherwise, there is a cut edge $xy$, 
where $y$ is not a leaf. 
By applying Lemma~\ref{cut} 
to $G$ with the cut edge $xy$ and using the minimality
of $G$, we get a contradiction.

Consider $H=G-\{x\in V(G):\deg(x)=1\}$,
which clearly has the same
girth as $G$ since the deletion doesn't break any cycle.
And we have $\delta(H)\geq2$,
otherwise $G$ has a vertex $x$ adjacent to $\deg(x)-1$ vertices of degree 1,
which is impossible.
Lemma \ref{p_lg_path} claims that there
is a path $x_{0}x_{1}\dots x_{\ell+1}$ with $\ell=\sigma+3$
and $\deg_{H}(x_{i})=2$ for $i=1,2,\dots,\ell$.
Now let $G'$ be subgraph obtained from $G$ by deleting the leaf-neighbors of $x_{2},x_{3},\dots,x_{\ell-1}$
and the vertices $x_{3},x_{4},\dots,x_{\ell-2}$.
Consider the subgraph $T$ of $G$ induced by $x_{1},x_{2},\dots,x_{\ell}$
and their neighbors, which is a caterpillar tree.
By Lemma \ref{caterpillar} that will be
proved in the next section, $T$ admits a  strong $\sigma$-edge-coloring $\varphi'$ such that
$\varphi$ and $\varphi'$ coincides on the edges incident to $x_1$ and $x_\ell$. Gluing
these two edge-colorings we construct a strong $\sigma$-edge-coloring of $G$. \qed

%%%%%%%%%%%%%%%%%%%%%%%%%%%%%%%%%%%%%%%%%%%%%%%%%%%%%%%%%%%%%%
\section{The key lemma: caterpillar with edge pre-coloring}
%%%%%%%%%%%%%%%%%%%%%%%%%%%%%%%%%%%%%%%%%%%%%%%%%%%%%%%%%%%%%%

All the graphs in this section are caterpillar trees.
Let $d_i\geq 2$ for $i=1,2,\dots,\ell$.
By $T=\Cat(d_{1},d_{2},\dots,d_{\ell})$ we mean a
caterpillar tree with spine $x_{0},x_{1},\dots,x_{\ell+1}$, whose
degrees are $d_{0},d_1,\dots,d_{\ell+1}$, where $d_{0}=d_{\ell+1}=1$.
Call $\ell$ the length of $T$ and
let $E_{i}$ be the edges incident with $x_{i}$. See Figure~\ref{fig:caterpillar}
for Cat(5,3,2,4,5).
\begin{figure}[!ht]
  \centering
  \includegraphics{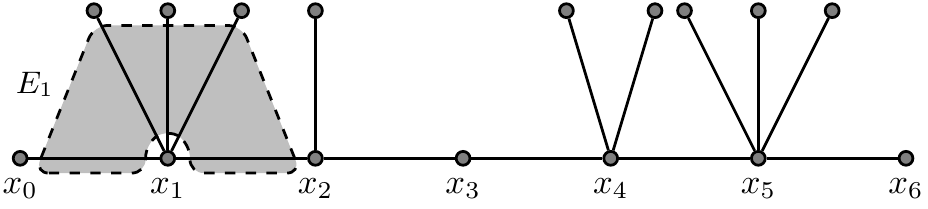}
  \vskip -0.4cm
  \caption{The caterpillar tree Cat(5,3,2,4,5).} \label{fig:caterpillar}
\end{figure}

Collect all the tuples $(C;\alpha_0,C_1,C_\ell,\alpha_\ell)$ as $\mathcal{P}_\kappa(T)$,
where the color sets $C_1, C_\ell\subseteq C$ with $|C_1|=d_1, |C_\ell|=d_\ell$, $|C|=\kappa$,
and $\alpha_0\in C_1, \alpha_\ell\in C_\ell$.
Fix $\kappa\in \mathbb{N}$. For any
$P=(C;\alpha_0,C_1,C_\ell,\alpha_\ell)\in \mathcal{P}_\kappa(T)$,
the set of all strong edge-colorings $\varphi$ using the colors in $C$ and
satisfying the following criterions
is denoted by $\mathcal{C}_T(P)$:
$$
  \varphi(E_1) = C_1, ~~~
  \varphi(E_\ell) = C_\ell, ~~~
  \varphi(x_0x_1)=\alpha_0 ~~~ {\rm and} ~~~
  \varphi(x_{\ell} x_{\ell+1})=\alpha_\ell.
$$

%\begin{defn}
%For a caterpillar tree $T$,
If $\mathcal{C}_T(P)$ is nonempty for any $P\in\mathcal{P}_\kappa(T)$ with $\kappa\geq\sigma(T)$,
then $T$ is called \emph{$\kappa$-two-sided strong edge-pre-colorable}.
%or \emph{$\kappa$-t.s.\ edge-pre-colorable} for abbreviation.
%\end{defn}
%We only consider $\kappa\geq \sigma(T)$, since any strong edge-coloring
%uses at least $\sigma(T)$ colors.

%-------------------------------------
\begin{lem} \label{kappaR}
If $T=\Cat(d_1,d_2,\dots,d_\ell)$ is $\kappa$-two-sided strong edge-pre-colorable,
then $T$ is $\kappa'$-two-sided strong edge-pre-colorable for any $\kappa'\geq \kappa$.
\end{lem}

\begin{proof}
For any $P'=(C';\alpha'_0,C'_1,C'_\ell,\alpha'_\ell)\in \mathcal{P}_{\kappa'}(T)$,
we have to find a strong edge-coloring in
$\mathcal{C}_T(P')$.

\smallskip
\textbf{Case $|C'_1\cup C'_\ell|\leq \kappa$.}
Choose a $\kappa$-set $C$ so that
$C'_1\cup C'_\ell\subseteq C\subseteq C'$. By assumption, there is a strong edge-coloring in
$\mathcal{C}_T(C;\alpha'_0,C'_1,C'_\ell,\alpha'_\ell)\subseteq \mathcal{C}_T(P')$.

\smallskip
\textbf{Case $|C'_1\cup C'_\ell|> \kappa$.}
Choose a $\kappa$-set $C$ so that $C'_1\cup \{\alpha'_\ell\} \subseteq C\subseteq C'_1\cup C'_\ell$,
and a $d_\ell$-set $C_\ell$ so that $C'_\ell \cap C \subseteq C_\ell \subseteq C$.
By assumption, there is a strong edge-coloring $\varphi$ in $\mathcal{C}_T(C;\alpha'_0,C'_1,C_\ell,\alpha'_\ell)$.
Let the edges in $E_\ell$ with color $C_\ell-C'_\ell$ be $E'_\ell$.
Notice $C'_\ell-C_\ell$ and $C$ are disjoint, so the colors in $C'_\ell-C_\ell$ are not appeared in $\varphi$.
Hence we can change the colors of $E'_\ell$
to $C'_\ell-C_\ell$ and obtain a strong edge-coloring in $\mathcal{C}_T(P')$.
\end{proof}
%-------------------------------------

We now derive a series of properties regarding the two-sided strong edge-pre-colorability
of a caterpillar tree and its certain subtrees.

\begin{lem} \label{subgraphR}
Suppose a caterpillar tree $\widetilde{T}$ contains $T$ as a subgraph, and
both have the same length.
If $\widetilde{T}$ is $\kappa$-two-sided strong edge-pre-colorable,
then $T$ is also $\kappa$-two-sided strong edge-pre-colorable.
\end{lem}

\begin{proof}
  Suppose $(C;\alpha_0,C_1,C_\ell,\alpha_\ell)\in \mathcal{P}_\kappa(T)$. We find
  $(C;\alpha_0,C'_1,C'_\ell,\alpha_\ell)\in \mathcal{P}_\kappa(\widetilde{T})$ such that
  $C'_1\supseteq C_1$ and $C'_\ell\supseteq C_\ell$.
  The lemma follows that
  any $\varphi'\in \mathcal{C}_{\widetilde{T}}(C;\alpha_0,C'_1,C'_\ell,\alpha_\ell)$ has a
  restriction $\varphi$ on $T$ so that $\varphi$
  is a strong edge-coloring in $\mathcal{C}_T(C;\alpha_0,C_1,C_\ell,\alpha_\ell)$.
\end{proof}

For $T=\Cat(d_1,d_2,\dots,d_\ell)$,
let $T_{-1}$ be the subtree $\Cat(d_{1},d_{2},\dots,d_{\ell-1})$.

\begin{lem} \label{lenR}
For $T=\Cat(d_1,d_2,\dots,d_\ell)$,
if $T_{-1}$ is $\kappa$-two-sided strong edge-pre-colorable,
where $\kappa\geq \sigma(T)$,
then so is $T$.
\end{lem}

\begin{proof}
For any $P=(C;\alpha_0,C_1,C_\ell,\alpha_\ell)\in \mathcal{P}_\kappa(T)$,
pick $\alpha_{\ell-1} \in C_\ell-\alpha_\ell$
and $C_{\ell-1}$ a $d_{\ell-1}$-subset of $C$ with
$C_{\ell-1} \cap C_\ell =\{\alpha_{\ell-1}\}$.
Notice that $C_{\ell-1}$ can be chosen since $d_{\ell-1}+d_{\ell}-1\leq \sigma(T) \leq \kappa$.
By the assumption, $T_{-1}$ admits a strong $\kappa$-edge-coloring
$\varphi\in\mathcal{C}_{T_{-1}}(C;\alpha_0,C_1,C_{\ell-1},\alpha_{\ell-1})$.
Coloring the remaining edges with $C_\ell-\alpha_{\ell-1}$ so that $x_{\ell}x_{\ell+1}$
has color $\alpha_\ell$
results in a
strong $\kappa$-edge-coloring in $\mathcal{C}_T(P)$.
\end{proof}

Hereafter, if necessary we reverse the order to view
$T=\Cat(d_{\ell},d_{\ell-1},\ldots,d_1)$ so that
we can always assume $\sigma(T_{-1})=\sigma(T)$.
Hence the requirement $\kappa\geq \sigma(T)$ in Lemma~\ref{lenR}
automatically holds.

For a caterpillar tree $T$, we define $T'$ and $I_T$ as follows.
Call a vertex $x_i$ \emph{$\sigma$-large} if $d_i\geq d^* := \lceil\frac{\sigma+1}{2}\rceil$.
The value $d^*$ is critical in the sense that
\begin{enumerate}
  \item If $d_i+d_j\leq\sigma+1$, then either $d_i$ or $d_j$ must be at most $d^*$.
  \item If $d_i+d_j\geq\sigma+1$, then either $d_i$ or $d_j$ must be at least $d^*$.
\end{enumerate}

Let $S=\{x_i:i\in I_T\}$ be the set of all $\sigma$-large vertices,
except that if there exist $i<j$
with $d_{i-1}<d^*$, $d_i=d_{i+1}=\ldots=d_j=d^*$ and $d_{j+1}<d^*$,
we only take $x_i,x_{i+2},x_{i+4},\ldots$ till $x_j$ or $x_{j-1}$,
depending on the parity. Then $S$ is a nonempty independent set.
Consider a new
degree sequence $d'_1,d'_2,\dots,d'_{\ell}$ where
\[
  d'_i=\begin{cases}
         d_i-1, &\text{ if } i\in I_T;\\
         d_i,   &\text{ if } i\notin I_T.
       \end{cases}
\]
Then $T' = \Cat(d'_{1},d'_{2},\dots,d'_{\ell})$ is
a caterpillar tree isomorphic to a subgraph of $T$,
with $\sigma(T')=\sigma(T)-1$ due to
the criticalness of $d^*$ and the choice method of $S$.
%that $d_i+d_{i+1}-1=\sigma$ implies $d_i \ge d^*$ or $d_{i+1} \ge d^*$,
%that is $x_i \in S$ or $x_{i+1} \in S$.

It is straightforward to see that
$T'_{-1} = \Cat(d'_{1},d'_{2},\dots,d'_{\ell-1})$ by the choice method of $S$.

\begin{lem} \label{caterpillar_induct}
For $T=\Cat(d_1,d_2,\dots,d_\ell)$,
suppose $\sigma(T)\geq 6$ and $T'_{-1}$ is $(\kappa-1)$-two-sided strong edge-pre-colorable,
then $T$ is $\kappa$-two-sided strong edge-pre-colorable.
\end{lem}

\begin{proof}
For any $P=(C;\alpha_0,C_1,C_\ell,\alpha_\ell)\in \mathcal{P}_\kappa(T)$,
we have to show that
$\mathcal{C}_T(P)$ is nonempty.

Let $I=I_T$. Our strategy is to search for a color $\beta$ such that
$$\mbox{
         $\beta\in C_1$ if and only if $1\in I$; and
         $\beta\in C_\ell$ if and only if $\ell\in I$.
}$$
%If such a color $\beta$ exists, we call
%$(T,I,C_1,C_\ell,\beta)$ a \emph{$\sigma$-reduction tuple}.
Suppose such a color $\beta$ exists and $\beta\neq \alpha_\ell$.
By Lemma~\ref{lenR}, $T'$ admits a strong $(\kappa-1)$-edge coloring
in $\mathcal{C}_{T'}(C-\beta;\alpha_0,C_1-\beta,C_\ell-\beta,\alpha_\ell)$.
Coloring the remaining edges with $\beta$ then yields the required
strong $\kappa$-edge-coloring in $\mathcal{C}_T(P)$.
Notice that $S$ being an independent set guarantees that the edges with color $\beta$
form an induced matching.
If it happens that $\beta$ coincides with $\alpha_\ell$,
then we seek instead for strong-edge coloring in
$\mathcal{C}_{T'}(C-\beta;C_1-\beta,\alpha_0,C_\ell-\beta,\alpha'_{\ell})$ for arbitrary
$\alpha'_{\ell}\in C_\ell-\alpha_\ell$.
We make use of the symmetry of pendant edges incident with $x_\ell$
and still achieve the goal.

Sometimes there is no suitable $\beta$. We alternatively consider $T_{-1}$.
By finding appropriate $d_{\ell-1}$-subset $C_{\ell-1}\subseteq C$ and
$\alpha_{\ell-1}$ with $C_{\ell-1}\cap C_{\ell} = \{\alpha_{\ell-1}\}$,
there will be a $\beta$ such that
$$\mbox{
         $\beta\in C_1$ if and only if $1\in I$; and
         $\beta\in C_{\ell-1}$ if and only if $\ell-1\in I$.
}$$
Similarly, there is a strong edge-coloring in
$\mathcal{C}_{T_{-1}}(C;\alpha_0,C_1,C_{\ell-1},\alpha_{\ell-1})$.
Color the remaining edges with $C_\ell-\alpha_{\ell-1}$ so that $x_{\ell}x_{\ell+1}$
has color $\alpha_\ell$, we gain a
strong $\kappa$-edge-coloring in $\mathcal{C}_T(P)$.

%We divide the situation as follows and show that the strategy above works.
We now prove the existence of $\beta$ according to the following four cases.

\smallskip
\textbf{Case 1.} $1,\ell\in I$.
In this case, $C_1\cap C_\ell$ is nonempty since
\[
   |C_1\cap C_\ell|=|C_1|+|C_\ell|-|C_1\cup C_\ell| \geq 2d^*-\sigma > 0.
\]
Pick $\beta$ to be any color in the intersection.

\smallskip
\textbf{Case 2.} $1\in I$ but $\ell\notin I$.
If $C_1-C_\ell$ is nonempty, then pick $\beta$ to be any color
in the difference.
Otherwise, $1\in I$ and $\ell\notin I$ imply $d_1\geq d^*\geq d_\ell$.
On the other hand,
$C_1-C_\ell = \emptyset$ implies $d_1\leq d_\ell$.
Thus the situation that $C_1-C_\ell$ is empty
occurs only when $d_1=d_\ell=d^*$ and $C_1=C_\ell$.
We consider the subtree $T_{-1}$.
Choose $\alpha_{\ell-1}$ to be any color in $C_\ell-\alpha_\ell$.
Let $C_{\ell-1}$ be $\alpha_{\ell-1}$ together with any
$(d_{\ell-1}-1)$-subset in $C-C_\ell$.

Since $d_\ell=d^*$ but $\ell \notin I$,
it is the case that $\ell-1 \in I$ and $d_{\ell-1}=d^*$.
Pick $\beta=\alpha_{\ell-1}$.

%Hence $(T^*,I,C_1,C_{\ell-1},\beta)$
%is a $\sigma$-reduction tuple.

%Suppose $\sigma(T_2)<\sigma(T)$. We append a pendant edge on each
%degree 2 vertex, and append some pendant edges on $x_i$ for $i=2,3,\dots,\ell-1$
%so that $T_3$ is a nice caterpillar.
%
%Apply $\sigma$-reduction to $T^*$ and obtain $T'$.
%Notice that compared to $T$, the length of $T'$ is less than one.
%However, $T'$ is still a nice tree since $\sigma(T')$ is also less than one.
%So $T^*$ admits an edge-coloring in
%$\mathcal{C}(T^*,C;\alpha_0,C_1,C_{\ell-1},\alpha_{\ell-1})$.
%Color the edges in $E(T^*)-E(T)$ by $C_\ell-\alpha_{\ell-1}$ properly resulting in an edge-coloring in
%$\mathcal{C}_T(C;\alpha_0,C_1,C_\ell,\alpha_\ell)$.

\smallskip
\textbf{Case 3.} $\ell\in I$ but $1\notin I$.
If $C_\ell-C_1$ is nonempty, then let $\beta$ be any color
in the difference.
Otherwise, $d_1=d_\ell=d^*$ and $C_1=C_\ell$. But $d_1=d^*$ implies
$1\in I$, a contradiction.

\smallskip
\textbf{Case 4.} $1,\ell\notin I$.
If $C-(C_1 \cup C_\ell)$ is nonempty, then pick $\beta$ to be any color
in the difference set.
Now, suppose $C=C_1 \cup C_\ell$. We consider the subtree $T_{-1}$.
%
%Similar to Case 2, it suffices to find a $\sigma$-reduction tuple for
%$T^*$.

First estimate the size
$$
|C_\ell-C_1|=|C_\ell\cup C_1|-|C_1|\geq \sigma-d^*\geq d^*-2 \geq 2,
$$
where $d^*\geq 4$ since $\sigma\geq 6$.
Pick $\alpha_{\ell-1}$ to be any color in $C_\ell-C_1-\alpha_\ell$. Let
$C_{\ell-1}$ be a color set such that
$|C_{\ell-1}|=d_{\ell-1}$ and $C_{\ell-1}\cap C_\ell = \{\alpha_{\ell-1}\}$.

When ${\ell-1}\in I$,
pick $\beta = \alpha_{\ell-1}$.
Otherwise, let $\beta$ be chosen from $C_\ell-C_1-\alpha_{\ell-1}$.
%Then $(T^*,I,C_1,C_{\ell-1},\beta)$ is one desired.
\end{proof}

%\begin{defn} \label{nice}
%A caterpillar tree $T$ is called \emph{nice} if
%$$
%  \sigma=\sigma(T)\geq5, ~~~
%  \ell\geq \max\{8,\sigma+3\} ~~~ {\rm and} ~~~
%  \sigma\geq\Delta(T)+2,
%$$
%and is called \emph{semi-nice} if only the first two conditions hold.
%\end{defn}

Now we are ready to prove the key lemma.
%Here is the key lemma.
%-------------------------------------------------
\begin{lem} \label{caterpillar}
Suppose $T=\Cat(d_{1},d_{2},\dots,d_{\ell})$ is a nice caterpillar tree,
i.e.\ it satisfies
$$
  \sigma=\sigma(T)\geq5, ~~~
  \ell\geq \sigma+3 ~~~ {\rm and} ~~~
  \sigma\geq\Delta(T)+2.
$$
For any $\kappa\geq \sigma(T)$,
any color sets $C_1, C_\ell \subseteq C$ with $|C|=\kappa$, $|C_1|=d_1$, $|C_\ell|=d_\ell$,
and any two colors $\alpha_0\in C_1$, $\alpha_\ell\in C_\ell$,
there is a strong $\sigma$-edge colorings $\varphi$ using the colors in $C$ such that
$\varphi(E_1) = C_1$, $\varphi(E_\ell) = C_\ell$
and $\varphi(x_0x_1)=\alpha_0$, $\varphi(x_{\ell} x_{\ell+1})=\alpha_\ell$.
That is,
$T$ is $\kappa$-two-sided strong edge-pre-colorable for any $\kappa\geq \sigma$.
\end{lem}

\begin{proof}
We prove the lemma by induction on $\sigma=\sigma(T)$.
By Lemmas~\ref{kappaR} and~\ref{lenR}, it suffices to consider the case $\kappa=\sigma$
and $\ell=\ell_\sigma$.

If $T$ is nice and $\sigma\geq 6$,
then $T'_{-1}$ is also a nice caterpillar tree:
The first two conditions remain since
$\sigma(T'_{-1})=\sigma(T')=\sigma(T)-1$.
The third one $\sigma(T'_{-1})\geq\Delta(T'_{-1})+2$
fails only when $\sigma(T)=\sigma(T'_{-1})+1\leq\Delta(T'_{-1})+2\leq \Delta(T)+2$
and so
$\Delta(T') = \Delta(T)$.
Since $\Delta(T) \ge d^*$, in this case,
$\Delta(T)=d^*\ge 4$ and there is at least a pair of $d_i=d_{i+1}=d^*$.
Then $\sigma(T'_{-1})=\sigma(T)-1=2\Delta(T)-2\geq \Delta(T)+2 \ge \Delta(T'_{-1})+2$.

%First, we may assume $\ell=\sigma+3$.
%For otherwise, suppose $\ell > \sigma+3$.
%Choose $\alpha_{\ell-1} \in C_\ell-\{\alpha_\ell\}$
%and $C_{\ell-1}$ with $|C_{\ell-1}|=d_{\ell-1}$
%and $C_{\ell-1} \cap C_\ell =\{\alpha_{\ell-1}\}$.
%We then reduce the proof to $\mathcal{C}(T^*,C;\alpha_0,C_1,C_{\ell-1},\alpha_{\ell-1})$.

By Lemma~\ref{caterpillar_induct}, we only have to discuss the base cases $\sigma=5$ and $\ell=8$.
We may assume all degrees $d_i=3$ since $\sigma\geq\Delta+2$.
Also assume $C_1=\{1,2,3\}$ and $\alpha_0=1$.
Depending on $C_1 \cap C_8$ and whether $\alpha_8 =  \alpha_0$ or not,
by symmetry we color $T$ according to $\varphi$ shown in Table~\ref{tbl:S5},
where $\alpha_i=\varphi(x_ix_{i+1})$ and $\widehat{C}_i=\varphi(C_i)-\varphi(x_{i-1}x_i)-\varphi(x_ix_{i+1})$.
Or we can solve this case by the argument in~\cite{Borodin+13} or the odd graph method in~\cite{Chang+13,WZ15}.
\end{proof}
\begin{table}[!ht]
\begin{center}
  \setlength{\tabcolsep}{5pt}
  \begin{tabular}{ccccccccccccccccc}
  $\alpha_0$ & $\widehat{C}_1$ & $\alpha_1$ & $\widehat{C}_2$ & $\alpha_2$ & $\widehat{C}_3$ & $\alpha_3$ & $\widehat{C}_4$ &
  $\alpha_4$ & $\widehat{C}_5$ & $\alpha_5$ & $\widehat{C}_6$ & $\alpha_6$ & $\widehat{C}_7$ & $\alpha_7$ & $\widehat{C}_8$ & $\alpha_8$ \\
  \hline
  1 & \{3\} & 2 & \{4\} & 5 & \{1\} & 3 & \{4\} & 2 & \{5\} & 1 & \{3\} & 4 & \{5\} & 2 & \{1\} & 3 \\
  1 & \{3\} & 2 & \{4\} & 5 & \{1\} & 3 & \{4\} & 2 & \{5\} & 1 & \{3\} & 4 & \{5\} & 2 & \{3\} & 1 \\
  1 & \{2\} & 3 & \{5\} & 4 & \{1\} & 2 & \{5\} & 3 & \{1\} & 4 & \{2\} & 5 & \{1\} & 3 & \{4\} & 2 \\
  1 & \{2\} & 3 & \{5\} & 4 & \{1\} & 2 & \{5\} & 3 & \{1\} & 4 & \{2\} & 5 & \{1\} & 3 & \{2\} & 4 \\
  1 & \{3\} & 2 & \{4\} & 5 & \{1\} & 3 & \{4\} & 2 & \{5\} & 1 & \{4\} & 3 & \{5\} & 2 & \{4\} & 1 \\
  1 & \{3\} & 2 & \{4\} & 5 & \{1\} & 3 & \{4\} & 2 & \{5\} & 1 & \{4\} & 3 & \{5\} & 2 & \{1\} & 4 \\
  1 & \{3\} & 2 & \{4\} & 5 & \{1\} & 3 & \{2\} & 4 & \{5\} & 1 & \{2\} & 3 & \{5\} & 4 & \{1\} & 2 \\
  1 & \{3\} & 2 & \{4\} & 5 & \{1\} & 3 & \{4\} & 2 & \{5\} & 1 & \{4\} & 3 & \{2\} & 5 & \{4\} & 1 \\
  1 & \{3\} & 2 & \{4\} & 5 & \{1\} & 3 & \{4\} & 2 & \{5\} & 1 & \{4\} & 3 & \{2\} & 5 & \{1\} & 4 \\
  1 & \{3\} & 2 & \{4\} & 5 & \{1\} & 3 & \{2\} & 4 & \{1\} & 5 & \{2\} & 3 & \{1\} & 4 & \{5\} & 2 \\
  1 & \{3\} & 2 & \{4\} & 5 & \{1\} & 3 & \{2\} & 4 & \{1\} & 5 & \{2\} & 3 & \{1\} & 4 & \{2\} & 5
  \end{tabular}
\end{center}
\vskip -0.5cm
\caption{The 5-strong edge-colorings of $T$ for $\sigma=5$ with $\ell=8$.}
\label{tbl:S5}
\end{table}

%\begin{cor} \label{caterpillar_cor}
%Suppose $T$ is a caterpillar tree of length $\ell$ satisfying
%$$
%  \sigma=\sigma(T)\geq 4 ~~~ {\rm and} ~~~
%  \ell\geq \sigma+4,
%$$
%then $T$ is $\kappa$-two-sided strong edge-pre-colorable for any $\kappa\geq \sigma+1$.
%\end{cor}
%
%\begin{proof}
%Add pendant edges at some vertices of $T$ with degree $\delta(T)$
%so that the resulting graph $\widetilde{T}$ has
%$\sigma(\widetilde{T})=\sigma(T)+1$ and $\sigma(\widetilde{T})\geq\Delta(\widetilde{T})+2$.
%Hence $\widetilde{T}$ is
%a nice caterpillar tree and is $\kappa$-two-sided strong edge-pre-colorable for any
%$\kappa\geq \sigma(\widetilde{T})=\sigma(T)+1$.
%The corollary then follows from Lemma~\ref{subgraphR}.
%\end{proof}

%%%%%%%%%%%%%%%%%%%%%%%%%%%%%%%%%%%%%%%%%%%%%%%%%%%%%%%%%%%%%%%%%%%%%%%%
\section{Refinement of Lemma~\ref{caterpillar}}\label{sec_refine}
%%%%%%%%%%%%%%%%%%%%%%%%%%%%%%%%%%%%%%%%%%%%%%%%%%%%%%%%%%%%%%%%%%%%%%%%

We now discuss the optimality of Lemma~\ref{caterpillar}.
If we take more care about the base cases, there would be a refinement:

\begin{lem} \label{caterpillar2}
Suppose $T$ is a caterpillar tree of length $\ell$ satisfying
$$
  \sigma=\sigma(T)\geq5, ~~~
  \ell\geq \ell_\sigma ~~~ {\rm and} ~~~
  \sigma\geq\Delta(T)+2,
$$
where
$$
\ell_\sigma =
\begin{cases}
  8, & \mbox{if } \sigma=5; \\
  7, & \mbox{if } \sigma=6,7; \\
  \sigma, & \mbox{if } \sigma\geq 8.
\end{cases}
$$
Then $T$ is $\kappa$-two-sided strong edge-pre-colorable for any $\kappa\geq \sigma$.
\end{lem}

\begin{proof}
Similar to Lemma~\ref{caterpillar}, we only need to consider the base cases.

For $\sigma=6$, we first consider the situation $\ell=6$.
By Lemma~\ref{subgraphR} and the symmetry,
it suffices to discuss the caterpillar trees
$\Cat(4,3,4,3,4,3)$,
$\Cat(4,3,4,3,3,4)$, and
$\Cat(3,4,3,3,4,3)$.
We enumerate all the cases in Table~\ref{tbl:S6_434343} and Table~\ref{tbl:S6_434334}
to show that the first two are $6$-two-sided strong edge-pre-colorable.
%
%\FloatBarrier
\begin{table}[!ht]
\begin{center}
  \begin{tabular}{ccccccccccccc}
  $\alpha_0$ & $\widehat{C}_1$ & $\alpha_1$ & $\widehat{C}_2$ & $\alpha_2$ & $\widehat{C}_3$ & $\alpha_3$ & $\widehat{C}_4$ &
  $\alpha_4$ & $\widehat{C}_5$ & $\alpha_5$ & $\widehat{C}_6$ & $\alpha_6$ \\
  \hline
1 & \{3 , 4\} & 2 & \{6\} & 5 & \{3 , 4\} & 1 & \{2\} & 6 & \{4 , 5\} & 3 & \{2\} & 1 \\
1 & \{2 , 4\} & 3 & \{5\} & 6 & \{1 , 4\} & 2 & \{3\} & 5 & \{4 , 5\} & 1 & \{3\} & 2 \\
1 & \{3 , 4\} & 2 & \{6\} & 5 & \{3 , 4\} & 1 & \{2\} & 6 & \{3 , 4\} & 5 & \{2\} & 1 \\
1 & \{2 , 4\} & 3 & \{5\} & 6 & \{1 , 4\} & 2 & \{5\} & 3 & \{4 , 6\} & 1 & \{5\} & 2 \\
1 & \{2 , 4\} & 3 & \{5\} & 6 & \{2 , 4\} & 1 & \{5\} & 3 & \{4 , 6\} & 2 & \{1\} & 5 \\
1 & \{2 , 4\} & 3 & \{5\} & 6 & \{2 , 4\} & 1 & \{5\} & 3 & \{2 , 4\} & 6 & \{5\} & 1 \\
1 & \{2 , 3\} & 4 & \{5\} & 6 & \{2 , 3\} & 1 & \{5\} & 4 & \{2 , 3\} & 6 & \{1\} & 5 \\
1 & \{3 , 4\} & 2 & \{6\} & 5 & \{1 , 4\} & 3 & \{2\} & 6 & \{1 , 5\} & 4 & \{3\} & 2 \\
1 & \{2 , 3\} & 4 & \{5\} & 6 & \{1 , 3\} & 2 & \{5\} & 4 & \{1 , 6\} & 3 & \{5\} & 2 \\
1 & \{2 , 3\} & 4 & \{5\} & 6 & \{1 , 3\} & 2 & \{5\} & 4 & \{1 , 6\} & 3 & \{2\} & 5 \\
1 & \{2 , 3\} & 4 & \{5\} & 6 & \{1 , 3\} & 2 & \{5\} & 4 & \{1 , 3\} & 6 & \{5\} & 2 \\
1 & \{2 , 4\} & 3 & \{5\} & 6 & \{1 , 4\} & 2 & \{5\} & 3 & \{1 , 4\} & 6 & \{2\} & 5
  \end{tabular}
\end{center}
\vskip -0.5cm
\caption{The 6-strong edge-colorings for $T=\Cat(4,3,4,3,4,3)$.}
\label{tbl:S6_434343}
\end{table}
\begin{table}[!ht]
\begin{center}
  \begin{tabular}{ccccccccccccc}
  $\alpha_0$ & $\widehat{C}_1$ & $\alpha_1$ & $\widehat{C}_2$ & $\alpha_2$ & $\widehat{C}_3$ & $\alpha_3$ & $\widehat{C}_4$ &
  $\alpha_4$ & $\widehat{C}_5$ & $\alpha_5$ & $\widehat{C}_6$ & $\alpha_6$ \\
  \hline
1 & \{2 , 4\} & 3 & \{5\} & 6 & \{2 , 4\} & 1 & \{3\} & 5 & \{6\} & 4 & \{2 , 3\} & 1 \\
1 & \{3 , 4\} & 2 & \{5\} & 6 & \{1 , 4\} & 3 & \{2\} & 5 & \{6\} & 1 & \{3 , 4\} & 2 \\
1 & \{3 , 4\} & 2 & \{5\} & 6 & \{1 , 3\} & 4 & \{5\} & 2 & \{6\} & 3 & \{4 , 5\} & 1 \\
1 & \{2 , 4\} & 3 & \{6\} & 5 & \{1 , 2\} & 4 & \{3\} & 6 & \{2\} & 1 & \{4 , 5\} & 3 \\
1 & \{2 , 4\} & 3 & \{6\} & 5 & \{1 , 2\} & 4 & \{3\} & 6 & \{2\} & 1 & \{3 , 4\} & 5 \\
1 & \{3 , 4\} & 2 & \{5\} & 6 & \{3 , 4\} & 1 & \{5\} & 2 & \{3\} & 6 & \{4 , 5\} & 1 \\
1 & \{3 , 4\} & 2 & \{6\} & 5 & \{3 , 4\} & 1 & \{6\} & 2 & \{3\} & 5 & \{1 , 6\} & 4 \\
1 & \{3 , 4\} & 2 & \{6\} & 5 & \{1 , 3\} & 4 & \{6\} & 2 & \{3\} & 1 & \{4 , 6\} & 5 \\
1 & \{3 , 4\} & 2 & \{6\} & 5 & \{1 , 4\} & 3 & \{2\} & 6 & \{1\} & 4 & \{3 , 5\} & 2 \\
1 & \{2 , 4\} & 3 & \{6\} & 5 & \{1 , 2\} & 4 & \{3\} & 6 & \{1\} & 2 & \{3 , 4\} & 5 \\
1 & \{3 , 4\} & 2 & \{5\} & 6 & \{1 , 4\} & 3 & \{5\} & 2 & \{1\} & 6 & \{4 , 5\} & 3 \\
1 & \{3 , 4\} & 2 & \{6\} & 5 & \{1 , 3\} & 4 & \{6\} & 2 & \{1\} & 3 & \{4 , 6\} & 5
\end{tabular}
\end{center}
\vskip-0.5cm
\caption{The 6-strong edge-colorings for $T=\Cat(4,3,4,3,3,4)$.}
\label{tbl:S6_434334}
\end{table}

If the caterpillar tree $T$ considered with $\sigma=6$ and $\ell=7$ has $T_{-1}=\Cat(3,4,3,3,4,3)$,
then $T$ is a subtree of $\Cat(3,4,3,3,4,3,4)$.
We can assume $T=\Cat(3,4,3,3,4,3,4)$ by Lemma~\ref{subgraphR}.
Reverse the direction to see $T$ as $\Cat(4,3,4,3,3,4,3)$. Then the subtree
$T_{-1}=\Cat(4,3,4,3,3,4)$, which is
$6$-two-sided strong edge-pre-colorable.
Hence all the caterpillar trees with $\sigma=6$ and $\ell=7$ are $6$-two-sided strong edge-pre-colorable.

For $\sigma=7$ and $\ell=7$. It suffices to consider the caterpillar trees in Table~\ref{tbl:T7_7}.
\begin{table}[!ht]
\begin{center}
  \begin{tabular}{cc}
  $T$ & $T'_{-1}$ \\
  \hline
  $\Cat(3,5,3,5,3,5,3)$ & $\Cat(3,4,3,4,3,4)$ \\
  $\Cat(5,3,5,3,3,5,3)$ & $\Cat(4,3,4,3,3,4)$ \\
  $\Cat(5,3,3,5,3,5,3)$ & $\Cat(4,3,3,4,3,4)$ \\
  $\Cat(5,3,5,3,5,3,5)$ & $\Cat(4,3,4,3,4,3)$ \\
  $\Cat(5,3,3,5,3,3,5)$ & $\Cat(4,3,3,4,3,3)$ \\
  $\Cat(3,5,3,5,3,4,4)$ & $\Cat(3,4,3,4,3,3)$ \\
  $\Cat(5,3,5,3,4,4,4)$ & $\Cat(4,3,4,3,3,4)$ \\
  $\Cat(3,5,3,4,4,4,4)$ & $\Cat(3,4,3,3,4,3)$ \\
  $\Cat(5,3,4,4,4,4,4)$ & $\Cat(4,3,3,4,3,4)$ \\
  $\Cat(4,4,4,4,4,4,4)$ & $\Cat(3,4,3,4,3,4)$ \\
  $\Cat(3,5,3,4,4,3,5)$ & $\Cat(3,4,3,3,4,3)$ \\
  $\Cat(5,3,4,4,4,3,5)$ & $\Cat(4,3,3,4,3,3)$ \\
  $\Cat(4,4,3,5,3,4,4)$ & $\Cat(3,4,3,4,3,3)$ \\
  $\Cat(4,4,3,5,3,3,5)$ & $\Cat(3,4,3,4,3,3)$ \\
  \end{tabular}
\end{center}
\vskip -0.5cm
\caption{The caterpillar trees to be considered for $\sigma=7$ and $\ell=7$.}
\label{tbl:T7_7}
\end{table}

All the trees $T$ considered except $\Cat(3,5,3,4,4,4,4)$ and $\Cat(3,5,3,4,4,3,5)$ have
$T'_{-1}$ being $6$-two-sided strong edge-pre-colorable, so these $T$ are
$7$-two-sided strong edge-pre-colorable by Lemma~\ref{caterpillar_induct}.

If we see $\Cat(3,5,3,4,4,4,4)$ as $T=\Cat(4,4,4,4,3,5,3)$, then $T'_{-1}=\Cat(3,4,3,4,3,4)$ is
$6$-two-sided strong edge-pre-colorable.
Similarly, regard $\Cat(3,5,3,4,4,3,5)$ as $T=\Cat(5,3,4,4,3,5,3)$, then $T'_{-1}=\Cat(4,3,3,4,3,4)$
is $6$-two-sided strong edge-pre-colorable. So these two trees are also $7$-two-sided strong edge-pre-colorable by
Lemma~\ref{caterpillar_induct},
and hence all the caterpillar trees considered with $\sigma=7$ and $\ell=7$ are $7$-two-sided strong edge-pre-colorable.
\end{proof}

The $\ell_\sigma$ here cannot be reduced:
For $\sigma\geq 7$, consider $\ell=\sigma-1$ and $T=\Cat(d_1,d_2,\dots,d_\ell)$,
where $d_1,d_3\dots=\lfloor\frac{\sigma+1}{2}\rfloor$ and
$d_2,d_4\dots=\lceil\frac{\sigma+1}{2}\rceil$.

If $\sigma=2d-1$ is an odd integer, let $P=([1,\sigma];1,[1,d],[1,d],1)\in \mathcal{P}_\sigma(T)$.
Suppose there is some $\varphi\in\mathcal{C}_T(P)$. Let $C_i=\varphi(E_i)$.
Then $|C_{i+2}-C_{i}|=1$ for $i=1,2,\dots,\ell-2$.
So $|C_\ell-C_{2}|\leq d-2$. However, $C_1=C_\ell$ implies $|C_\ell-C_2|=d-1$, a contradiction.

If $\sigma=2d-2$ is an even integer, let $P=([1,\sigma];1,[1,d-1],[d,2d-2],d)\in \mathcal{P}_\sigma(T)$.
Suppose there is some $\varphi\in\mathcal{C}_T(P)$. Let $C_i=\varphi(E_i)$.
Again $|C_{i+2}-C_{i}|=1$ for $i=1,2,\dots,\ell-2$.
So $d-1=|C_\ell-C_1|\leq d-2$, a contradiction.

For $\sigma=6$, let $T=\Cat(3,4,3,3,4,3)$ and
$P=([1,6];1,\{1,2,3\},\{4,5,6\},6)\in \mathcal{P}_\sigma(T)$.
Suppose there is some $\varphi\in\mathcal{C}_T(P)$. Let $C_i=\varphi(E_i)$ and $X=C_3\cup C_4-\varphi(x_2x_3)-\varphi(x_4x_5)$.
Then $|C_1\cap X| = |X\cap C_6| = 2$ implies $|X|\geq 4$, which is impossible since $|X|=3$.

Exploiting Lemma~\ref{caterpillar2}, the main Theorem~\ref{main} can be strengthened to:

\begin{thm} \label{main2}
If $G$ is a planar graph with $\sigma=\sigma(G)\geq 5$,
$\sigma\geq\Delta(G)+2$ and girth at least $g_\sigma$,
where
$$
g_{\sigma} =
\begin{cases}
  41, & \mbox{if } \sigma=5; \\
  36, & \mbox{if } \sigma=6,7; \\
  5\sigma+1, & \mbox{if } \sigma\geq 8,
\end{cases}
$$
then $\chi_{s}'(G)=\sigma$.
\end{thm}

If we take off the condition $\sigma\geq\Delta+2$ in Theorem~\ref{main2},
a weaker result can be obtained by using the following corollary of
Lemma~\ref{caterpillar2} in the proof of  the main Theorem~\ref{main}.

\begin{cor} \label{caterpillar2_cor}
Suppose $T$ is a caterpillar tree of length $\ell$ satisfying
$$
  \sigma=\sigma(T)\geq 4 ~~~ {\rm and} ~~~
  \ell\geq \ell_{\sigma+1}, ~~~
$$
where
$$
\ell_{\sigma+1} =
\begin{cases}
  8, & \mbox{if } \sigma+1=5; \\
  7, & \mbox{if } \sigma+1=6,7; \\
  \sigma+1, & \mbox{if } \sigma+1\geq 8.
\end{cases}
$$
Then $T$ is $\kappa$-two-sided strong edge-pre-colorable for any $\kappa\geq \sigma+1$.
\end{cor}

\begin{proof}
Add pendant edges at some vertices of $T$ with degree $\delta(T)$
such that the resulting graph $\widetilde{T}$ has
$\sigma(\widetilde{T})=\sigma(T)+1$ and $\sigma(\widetilde{T})\geq\Delta(\widetilde{T})+2$.
So $\widetilde{T}$ satisfies the requirements of Lemma~\ref{caterpillar2},
and hence it is $\kappa$-two-sided strong edge-pre-colorable for any
$\kappa\geq \sigma(\widetilde{T})=\sigma(T)+1$.
The corollary then follows from Lemma~\ref{subgraphR}.
\end{proof}

\begin{thm} \label{main2_cor}
If $G$ is a planar graph with $\sigma=\sigma(G)\geq 4$ and girth at least $g_{\sigma+1}$,
where
$$
g_{\sigma+1} =
\begin{cases}
  41, & \mbox{if } \sigma+1=5; \\
  36, & \mbox{if } \sigma+1=6,7; \\
  5\sigma+6, & \mbox{if } \sigma+1\geq 8,
\end{cases}
$$
then $\sigma\leq\chi_{s}'(G)\leq\sigma+1$.
\end{thm}

%%%%%%%%%%%%%%%%%%%%%%%%%%%%%%%%%%%%%%%%%%%%%%%%%%%%%%%%%%%%%%%%%%
\section{Consequences concerning the maximum average degree}
%%%%%%%%%%%%%%%%%%%%%%%%%%%%%%%%%%%%%%%%%%%%%%%%%%%%%%%%%%%%%%%%%%

The following lemma is a direct consequence of Proposition 2.2 in~\cite{CW13}.

\begin{lem} \label{small_ad_path}
Suppose the connected graph $G$ is not a cycle. If $G$ has minimum degree at least $2$
and average degree $\frac{2|E|}{|V|}<2+\frac{2}{3\ell-1}$, then $G$
contains an $\ell$-thread.
\end{lem}

A $C_n$-jellyfish is a graph by adding pendant edges at the vertices of $C_n$. In~\cite{CC+2015},
it is shown that

\begin{prop} \label{cactus}
If $G$ is a $C_n$-jellyfish of $m$ edges with $\sigma(G) \ge 4$, then $\chi_s'(G)=$
\begin{eqnarray*}\hskip -1cm
&&\left\{\begin{array}{ll}
      m,             & \mbox{if $n=3$}; \\
      \sigma(G)+1,   & \mbox{if $n=4$}; \\
      \lceil \frac{m}{\lfloor n/2 \rfloor} \rceil,
                     & \mbox{otherwise, if $n$ is odd with all $\deg(v_i)=d$ but $(n,d) \ne (7,3)$,} \\
                     & \mbox{\hskip 1.9cm or with $\lceil \frac{m}{\lfloor n/2 \rfloor}\rceil\ge\sigma(G)+1$}; \\
      \sigma(G)+1,   & \mbox{otherwise, if $(n,d) = (7,3)$ with all $\deg(v_i)=d$},\\
                     & \mbox{\hskip 1.9cm or $n\not\equiv 0\ {\rm (mod}\ 3)$ such that up to rotation} \\
                     & \mbox{\hskip 2.4cm $\deg(v_i)=\sigma(G)-1$ for $i\equiv 1\ ({\rm mod\ 3})$ with $1 \le i \le 3\lfloor \frac{n}{3} \rfloor -2$,} \\
                     & \mbox{\hskip 1.9cm or $(n,\sigma(G))=(10,4)$ with $\deg(v_i)=3$ for all odd or all even $i$}; \\
      \sigma(G),     & \mbox{otherwise}.
\end{array} \right.
\end{eqnarray*}
\end{prop}

Adopting these results leads to a strengthening of Theorem~\ref{mad_WZ}.

\begin{thm}
If $G$ is a graph with $\sigma=\sigma(G)\geq 5$,
$\sigma\geq\Delta(G)+2$, odd girth at least $g'_\sigma$,
even girth at least 6,
and $\mad(G)<2+\frac{2}{3\ell_\sigma-1}$,
where
$$
g'_\sigma =
\begin{cases}
  9, & \mbox{if } \sigma=5; \\
  \sigma, & \mbox{if } \sigma>5,
\end{cases}
~~~
{\rm and}
~~~
\ell_\sigma =
\begin{cases}
  8, & \mbox{if } \sigma=5; \\
  7, & \mbox{if } \sigma=6,7; \\
  \sigma, & \mbox{if } \sigma\geq 8,
\end{cases}
$$
then $\chi_{s}'(G)=\sigma$.
\end{thm}

\begin{proof}
In the proof of Theorem~\ref{main2}, alternatively use Lemma~\ref{small_ad_path} to find
an $\ell_\sigma$-thread in $H$.
It should be noticed the girth constraints exist merely to address the problem
that $H$ may be a cycle. In this case, by Proposition~\ref{cactus},
$G$ still has strong chromatic index $\sigma$.

Indeed, suppose $H=C_n$ and $G$ is a $C_n$-jellyfish.
The case $n$ is even is trivial. If $\sigma\geq\sigma(H)\geq 5$, $n$ is odd and
$n\geq g'_\sigma\geq\sigma$, then
$$
\Bigl\lceil\frac{|E(G)|}{\lfloor\frac{n}{2}\rfloor}\Bigr\rceil
\leq \Bigl\lceil\frac{\frac{n-1}{2}(\sigma-1)+\frac{\sigma+1}{2}-1}{\frac{n-1}{2}}\Bigr\rceil
\leq \sigma.
$$
Hence $\chi'_s(G)=\sigma$.
\end{proof}

Similarly, Theorem~\ref{main2_cor} can be modified correspondingly.

\begin{thm}
If $G$ is a graph with $\sigma=\sigma(G)\geq 4$,
odd girth at least $\frac{\sigma+1}{2}$,
and $\mad(G)<2+\frac{2}{3\ell_{\sigma+1}-1}$,
where
$$
\ell_{\sigma+1} =
\begin{cases}
  8, & \mbox{if } \sigma+1=5; \\
  7, & \mbox{if } \sigma+1=6,7; \\
  \sigma+1, & \mbox{if } \sigma+1\geq 8,
\end{cases}
$$
then $\sigma\leq\chi_{s}'(G)\leq\sigma+1$.
\end{thm}

\bigskip
\noindent
{\bf Acknowledgements.}
This project was supported in part by the Ministry of Science
and Technology (Taiwan) under grant 104-2115-M-002-006-MY2.
The authors thank Tao Wang for extensive discussion and providing many useful comments.

%\nolinenumbers
% used to include all bibitems
%\nocite{*}
%\bibliography{SCI_P_LG}

\end{document}